\documentclass[11pt,twoside,a4paper]{article}
\usepackage{amsmath,amssymb,mathrsfs}
\usepackage{amsthm}
\usepackage[utf8x]{inputenc}
\usepackage{bm}
\usepackage{avant}
\usepackage[english]{babel}
\usepackage{thmtools,thm-restate}
\usepackage[nottoc]{tocbibind}
\usepackage{hyperref}
\usepackage{graphicx}
\usepackage{enumitem}
\usepackage{tikz}
\usepackage{bbm}
\usepackage[inner=2.4cm,outer=2.4cm,top=2.8cm,bottom=2.8cm]{geometry}
\usepackage{mathtools}
\usepackage{makeidx}
\usepackage{bold-extra}
\usepackage{amsfonts}

\newcommand{\overbar}[1]{\mkern 1mu\overline{\mkern-1mu#1\mkern-1mu}\mkern 1mu}

\DeclareMathOperator{\id}{id}

\newcommand{\modu}[1]{\left |#1\right |}
\newcommand{\Leb}{\mathscr{L}}
\newcommand{\setR}{\mathbb{R}}

\newcommand{\R}{\mathbb{R}}

\newcommand{\p}{\mathtt p} %projection
\newcommand{\de}{\ensuremath{\, \mathrm d}} % in the integrals

\newcommand{\suchthat}{\ensuremath{\,:\,}} % such that inside, for example, the sets definition
\newcommand\restr[2]{{% we make the whole thing an ordinary symbol
  \left.\kern-\nulldelimiterspace % automatically resize the bar with \right
  #1 % the function
  \right|_{#2} % this is the delimiter
  }}

\newcommand{\CD}{\mathsf{CD}}
\newcommand{\cd}{\mathsf{CD}}
\newcommand{\RCD}{\mathsf{RCD}}
\newcommand{\MCP}{\mathsf{MCP}}
\newcommand{\X}{\mathsf{X}}

\newcommand{\di}{\mathsf d} %standard mms distance notation
\newcommand{\m}{\mathfrak m} %standard mms measure notation 
\DeclareMathOperator{\Ent}{Ent}
\DeclareMathOperator{\Geo}{Geo}

\newcommand{\Prob}{\mathscr{P}}
\newcommand{\ProbTwo}{\mathscr{P}_2}
\usepackage{titlesec}
\titleformat{\section}
  {\scshape\large\centering}
  {\thesection}{0.5em}{}
\titleformat{\subsection}
  {\scshape\centering}
  {\thesubsection}{0.4em}{}  

\author{Mattia Magnabosco \footnote{Institute for Applied Mathematics - Universit\"at Bonn, \textit{magnabosco@iam.uni-bonn.de}}}
\title{\textbf{Examples of $\cd(0,N)$ spaces with non-constant dimension}}
\date{}

\newtheoremstyle{remark}% <name>
        {10pt}% <Space above>
        {10pt}% <Space below>
        {}% <Body font>
        {}% <Indent amount>
        {\itshape}% <Theorem head font>
        {.}% <Punctuation after theorem head>
        {.4em}% <Space after theorem head>
        {}% <Theorem head spec (can be left empty, meaning 'normal')>

\newtheoremstyle{proof}% <name>
        {10pt}% <Space above>
        {10pt}% <Space below>
        {}% <Body font>
        {}% <Indent amount>
        {\itshape}% <Theorem head font>
        {.}% <Punctuation after theorem head>
        {.4em}% <Space after theorem head>
        {}% <Theorem head spec (can be left empty, meaning 'normal')>
        
\newtheoremstyle{definition}% <name>
        {10pt}% <Space above>
        {10pt}% <Space below>
        {}% <Body font>
        {}% <Indent amount>
        {\bfseries}% <Theorem head font>
        {.}% <Punctuation after theorem head>
        {.4em}% <Space after theorem head>
        {}% <Theorem head spec (can be left empty, meaning 'normal')>        

\newtheoremstyle{theorem}% <name>
        {10pt}% <Space above>
        {10pt}% <Space below>
        {\slshape}% <Body font>
        {}% <Indent amount>
        {\bfseries}% <Theorem head font>
        {.}% <Punctuation after theorem head>
        {.4em}% <Space after theorem head>
        {}% <Theorem head spec (can be left empty, meaning 'normal')>

\theoremstyle{theorem}
\newtheorem{theorem}{Theorem}[section]
\newtheorem{prop}[theorem]{Proposition}
\newtheorem{corollary}[theorem]{Corollary}
\newtheorem{lemma}[theorem]{Lemma}

\theoremstyle{definition}
\newtheorem{definition}[theorem]{Definition}

\theoremstyle{remark}
\newtheorem{remark}[theorem]{Remark}
\newtheorem{example}[theorem]{Example}

\theoremstyle{proof}
\newtheorem*{pro}{Proof}
\newenvironment{pr}{\begin{pro}%
 \pushQED{\qed}}%
 {\popQED\end{pro}}

\makeindex
\begin{document}
\maketitle

\begin{abstract}
    In this work, we generalize the results obtained in \cite{MR4402722}, presenting some examples of $\cd(0,N)$ spaces having different dimensions in different regions, deducing in particular that the topological splitting may fail in $\cd(0,N)$ spaces. We also observe that any reasonable non-branching condition may fail in $\cd(0,N)$ spaces and that the existence of an optimal transport map, between two absolutely continuous marginals, is not guaranteed by the $\cd(0,N)$ condition, without requiring a non-branching assumption. Moreover, we show that the strict $\cd(0,N)$ condition is strictly stronger than the classical $\cd(0,N)$ one and it is not stable with respect to the measured Gromov-Hausdorff convergence.%\vspace{4pt} \\
%\textbf{Keywords}: optimal transport, CD spaces, measured Gromov Hausdorff convergence, essentially non-branching spaces, branching geodesics \vspace{4pt} \\
%\textit{The author have no relevant financial or non-financial interests to disclose.}
\end{abstract}\vspace{10pt}

\section{Introduction}

In their seminal papers \cite{sturm2006,sturm2006ii,lottvillani}, Sturm and Lott--Villani introduced the so-called $\cd(K,N)$ condition, a synthetic notion representing a lower
curvature bound by $K$ and an upper bound on the dimension by $N$, formulated in the non-smooth setting of metric measure spaces. Their works are based on the observation that, for a (weighted) Riemannian manifold, having Ricci curvature bounded below and dimension bounded above, can be equivalently characterized in terms of a convexity property of the R\'enyi entropy functional, along Wasserstein geodesics. In particular, this property relies on the theory of optimal transport and does not require the smooth underlying structure, therefore it can be taken as  definition of curvature dimension bound for a metric measure space.\\

In this paper we show different examples of $\cd(0,N)$ spaces (i.e. spaces satisfying the $\cd(0,N)$ condition), having different singularities in their metric measure structure. This work is a twofold generalization of \cite{MR4402722}, where an example of a highly branching $\cd(0,\infty)$ space with non-constant (topological) dimension is constructed. On the one hand, we extend the result of \cite{MR4402722} to $\cd(0,N)$ spaces, having a finite dimensional bound $N$. This generalization is somewhat expected but far from trivial, in fact the finite dimensional $\cd$ condition implies some properties which are not guaranteed in $\cd(0,\infty)$ spaces (for example the Bishop-Gromov inequality). On the other hand, we prove the $\cd(0,N)$ condition for a class of metric measure spaces which is considerably larger than the one considered in \cite{MR4402722}. This allows to highlight other types of singular behaviour that are proved to be possible in $\cd$ spaces.\\

The main conclusions, which can be drawn from the examples of $\cd$ spaces considered in this paper, are the following: 
\begin{itemize}
    \item The constancy of dimension may fail in $\cd(0,N)$ spaces, i.e. there exists a $\cd(0,N)$ space having different topological and Hausdorff dimensions in different regions (see Section \ref{sec:conc1}). This is particularly interesting if seen in relation with the work of Brué and Semola \cite{bruesemola}, that proves constancy of dimension for $\RCD(K,N)$ spaces.
    \item The strict $\cd(0,N)$ condition (see Definition \ref{def:strictCD}) is not stable with respect to the measured Gromov-Hausdorff convergence. Moreover, the strict $\cd(0,N)$ condition is strictly stronger than the classical $\cd(0,N)$ one, i.e. there exists a $\cd(0,N)$ space which does not satisfy the strict $\cd(0,N)$ condition (see Section \ref{sec:conc2}).
    \item Any reasonably meaningful non-branching condition may fail in $\cd(0,N)$ spaces (see Section \ref{sec:conc2}).
    \item The existence of an optimal transport map between two absolutely continuous marginals is not guaranteed in $\cd(0,N)$ spaces, without assuming a non-branching condition (see Section \ref{sec:conc2}).
    \item The topological splitting may fail in $\cd(0,N)$ spaces, i.e. there exists a $\cd(0,N)$ space containing a subset isometric to $\R$, which does not topologically split as the product of $\R$ with another space (see Section \ref{sec:conc3}).
\end{itemize} \vspace{8pt}

This work, as it is for \cite{MR4402722}, is inspired by the work of Ketterer and Rajala \cite{ketterer2014failure}, where similar conclusions were drawn for spaces satisfying the measure contraction property $\MCP(0,N)$. We consider metric measure measure spaces having metric structure and singularities similar to the ones consider in \cite{ketterer2014failure}, but equipped with a more complicated measure that allows to achieve the $\cd$ condition (cfr. \cite[Remark 1]{ketterer2014failure}). This idea was already developed in \cite{MR4402722}, were the $\cd(0,\infty)$ condition was proved in a space with non-constant topological dimension. In this paper, we improve the strategy of \cite{MR4402722}, extending it to a larger class of spaces and refining different estimates and computations to prove the $\cd(0,N)$ condition. This generalization relies on the relation between the $(0,N)$-convexity of interpolating densities along geodesics and the $\cd(0,N)$ condition (see \cite[Proposition 4.2]{sturm2006ii}). The combination of this observation with a suitable version of the Jacobi equation proved by Rajala in \cite{rajala2013failure}, that allows to compute interpolating densities, results in Proposition \ref{prop:jacobi}, which will be fundamental to prove the $\cd(0,N)$ condition.\\

\noindent {\scshape Aknowlegments}: The author is grateful to Prof. Karl-Theodor Sturm for many inspiring discussions on the topic.

\section{Preliminaries}

\subsection{CD spaces}\label{sec:CD}
A triple $(X,\di,\m)$ is called metric measure space if $(X,\di)$ is a complete and separable metric space and $\m$ is a locally finite Borel measure on it.
Denote by $\Prob(X)$ the set of Borel probability measures on $X$ and by $\Prob_2(X) \subset \Prob(X)$ the set of those having finite second moment. We endow the space $\Prob_2(X)$ with the Wasserstein distance $W_2$, defined by
\begin{equation}
\label{eq:defW2}
    W_2^2(\mu_0, \mu_1) := \inf_{\pi \in \mathsf{Adm}(\mu_0,\mu_1)}  \int \di^2(x, y) \, \de \pi(x, y),
\end{equation}
where $\mathsf{Adm}(\mu_0, \mu_1)$ is the set of all the admissible transport plans between $\mu_0$ and $\mu_1$, namely all the measures in $\Prob(X^2)$ such that $(\p_1)_\sharp \pi = \mu_0$ and $(\p_2)_\sharp \pi = \mu_1$. We say that $\pi \in \mathsf{Adm}(\mu_0,\mu_1)$ is an optimal transport plan between $\mu_0$ and $\mu_1$ if it realizes the infimum in \eqref{eq:defW2}. Moreover, we say that an optimal transport plan $\pi$ is induced by a map if there exists $T$ measurable such that $\pi=(\id,T)_\# \mu_0$, in this case $T$ is said to be an optimal transport map.
For every $N>1$, define the R\'enyi entropy functional 
\begin{equation*}
    S_N(\mu) = - \int \rho^{1-\frac{1}{N}} \de \m,
\end{equation*}
where $\rho$ denotes the density of the absolutely continuous part of $\mu$ with respect to the reference measure $\m$, i.e. $\mu= \rho \m + \mu^s$ with $\mu^s \perp \m$. Call $\Prob_{ac}(X,\m)$ the set of all probability measures in $\Prob_2(X)$ which are absolutely continuous with respect to the reference measure $\m$.

\begin{definition}\label{def:CD}
Given $N>1$, the metric measure space $(X,\di,\m)$ is said to be a $\cd(0,N)$ space (or to satisfy the $\cd(0,N)$ condition) if for every pair of measures $\mu_0,\mu_1 \in \Prob_{ac}(X,\m)$ there exists a constant speed $W_2$-geodesic $(\mu_t)_{t\in [0,1]}\subset \Prob_{ac}(X,\m)$ connecting them, along which the entropy functional $S_{N'}$ is convex for every $N'\geq N$, i.e. 
\begin{equation*}
    S_{N'}(\mu_t) \leq (1-t) S_{N'}(\mu_0) + t S_{N'}(\mu_1), \quad \forall t \in [0,1].
\end{equation*}
\end{definition}

\begin{remark}\label{rmk:W2geo}
    Every constant speed $W_2$-geodesic $(\mu_t)_{t\in [0,1]}$ can be represented by a probability measure $\eta$ on the space $\Geo(X)$ of constant speed geodesics of $X$ (parameterized on $[0,1]$), meaning that $(e_t)_\# \eta=\mu_t$ for every $t\in [0,1]$, where $e_t$ is the evaluation map at time $t$.
\end{remark}

\noindent As every convexity property, the $\cd(0,N)$ condition can be equivalently characterized just looking at midpoints instead of whole geodesics (cfr. \cite[Proposition 1.8]{MR4402722}):

\begin{prop}\label{prop:CDinmidpoint}
    The metric measure space $(X,\di,\m)$ is a $\CD(0,N)$ space if for every pair $\mu_0,\mu_1\in\Prob_{ac}(X,\m)$ there exists a midpoint $\nu\in\Prob_{ac}(X,\m)$ of $\mu_0$ and $\mu_1$, such that for every $N'\geq N$ 
\begin{equation*}\label{eq:midpointconv}
    S_{N'}(\nu) \leq \frac12 S_{N'}(\mu_0) + \frac12 S_{N'}(\mu_1).
\end{equation*}
\end{prop}

The $\cd$ condition can be defined also when the dimensional parameter takes the value $\infty$, requiring convexity of the entropy functional $\Ent:\ProbTwo(X)\to \R \cup \{+\infty\}$, defined as
\begin{equation}\label{eq:BSEnt}
    \Ent(\mu):= 
    \begin{cases}
    \int \rho \log \rho \de \m  &\text{if }\mu\ll\m \text{ and } \mu=\rho\m\\
    +\infty &\text{otherwise}
    \end{cases},
    \end{equation} 
which is the limit of $N+ N S_N$ as $N \to \infty$.
\begin{definition}\label{def:CDinfty}
A metric measure space $(X,\di,\m)$ is said to be a $\cd(0,\infty)$ space (or to satisfy the $\cd(0,\infty)$ condition) if for every pair of measures $\mu_0,\mu_1 \in \Prob_{ac}(X,\m)$ there exists a constant speed $W_2$-geodesic $(\mu_t)_{t\in [0,1]}$ connecting them, along which the entropy functional $\Ent$ is convex, i.e. 
\begin{equation*}
   \Ent(\mu_t) \leq (1-t) \Ent(\mu_0) + t \Ent(\mu_1), \quad \forall t \in [0,1].
\end{equation*}
\end{definition}

\begin{remark}\label{rmk:Ent}
    Observe that the entropy functionals $\Ent$ an $S_N$ (for some $N>1$) have different behaviour on singular measures, in fact $\Ent(\mu)=+\infty$ whenever $\mu\not\ll\m$, while the singular part of $\mu$ does not contribute to the value of $S_N(\mu)$. This substantial difference is the reason why, in the definition of the $\cd(0,\infty)$ condition there is no need to require the Wasserstein geodesic to be contained in $\Prob_{ac}(X,\m)$.
\end{remark}

\noindent The $\cd(K,N)$ condition can be defined for every real value of the dimensional parameter $K$, but its formulation becomes more complicated as it involves the distortion coefficients $\tau_{K,N}^{(t)}$. We refer to \cite{sturm2006ii} for the precise definition of the $\cd(K,N)$ condition (with $K\neq0$), this general notion will not be used in this work. Among other nice properties, the $\cd(K,N)$ condition has the monotonicity in the parameters that we expect from a requirement that represents a lower curvature bound and an upper dimensional bound, i.e.
\begin{equation*}
    \cd(K',N') \implies \cd(K'',N'') \quad \text{if }K'\geq K'' \text{ and }N'\leq N'',
\end{equation*}
for every $K',K''\in \R$ and $N',N''\in(1,\infty]$.

\subsection{$(K,N)$-convexity of functions}

Given $K\in \R$, $N>0$ and an interval $I\subseteq \R$, a function $g\in C^2\big(I, \R\big)$ is said to be $(K,N)$-convex if 
\begin{equation*}
    g''(x)\geq K + \frac 1N (g'(x))^2  \quad \text{for every }x\in I.
\end{equation*}
In working with $(K,N)$-convex is often convenient to use the following equivalent characterization: $g\in C^2\big(I, \R\big)$ is $(K,N)$-convex if and only if the function $g_N:= e^{-g/N}$ satisfies 
\begin{equation*}
    g_N''(x) \leq - \frac{K}{N} g_N.
\end{equation*}
As a consequence, we also deduce that:
\begin{equation}\label{eq:0Nconv=concavity}
    g\in C^2\big(I, \R\big) \text{ is }(0,N)\text{-convex if and only if $g_N$ is concave.}
\end{equation}

\begin{example} \label{ex:KNconv}
    The function $-\log:(0,\infty)\to \R$ is $(0,1)$-convex, while for every $K>0$ the function 
    \begin{equation*}
        [0,1] \ni x \mapsto K x^2
    \end{equation*}
    is $(0,2K)$-convex.
\end{example}

\begin{remark}\label{rmk:repar}
    The $(0,N)$-convexity is invariant by linear reparametrization, meaning that, if $g\in C^2\big(I, \R\big)$ is $(0,N)$-convex, then for every $\alpha\in \R$ and $\beta\neq 0$ the function 
    \begin{equation*}
        t \mapsto g( \alpha +\beta t)
    \end{equation*}
    is $(0,N)$-convex (where defined). More in general, if $g\in C^2\big(I, \R\big)$ is $(K,N)$-convex, then for every $\alpha\in \R$ and $\beta\neq 0$ the function $t \mapsto g( \alpha +\beta t)$ is $(\beta^2 K,N)$-convex.
\end{remark}

The $(K,N)$-convexity enjoys different nice properties, among these the following additivity for the constants will be particularly important in section \ref{section:CDproof}.
\begin{lemma}\label{lem:sumKN}
 If $g\in C^2\big(I, \R\big)$ is $(K_1,N_1)$-convex and $h\in C^2\big(I
 , \R\big)$ is $(K_2,N_2)$-convex, then $g+h$ is $(K_1+K_2,N_1+N_2)$-convex,
\end{lemma}
\noindent For a proof of this lemma we refer to \cite[Lemma 2.10]{MR3385639}.

We conclude the subsection with a technical lemma that will be useful to prove the main result (Theorem \ref{thm:CDTheExample}).

\begin{lemma}\label{lem:approxconv}
    Given $A\in [0,\infty)$ and $\delta \in \big(-\frac 1{2^{11}}, \frac 1{2^{11}}\big)$, there exists a function $h:[0,1] \to \R$ strictly positive on $[0,1)$, with $h(0)=1$, $h(1)=A$ and $h\big(\frac 12\big)= 1 + \big(\frac12+\delta\big)(A-1)$, such that $t\mapsto-\log(h(t))$ is $(- 2^{21}\, \delta^2,2)$-convex
    %$h(t)=1 + t(A-1)$ in $\big[0,\frac 14\big]\cup\big[\frac 34,1\big]$
\end{lemma}

\begin{pr}
    Let $\phi:[0,1]\to [0,1]$ be a $C^2$ function such that $\phi\big(\frac 12\big)=1$, $\phi(t)=0$ if $t\in\big[0,\frac 14\big]\cup\big[\frac 34,1\big]$ and $|\phi'|\leq 2^4,\,|\phi''|\leq 2^7$ on $[0,1]$. Define $h$ as 
    \begin{equation*}
        h(t)= 1 + t(A-1) + \delta \phi(t) (A-1),
    \end{equation*}
    observe that 
    \begin{equation*}
         h'(t)= (A-1) [1+ \delta \phi'(t)] \qquad \text{and}\qquad h''(t)= \delta \phi''(t) (A-1).
    \end{equation*}
    Now, we want to prove that 
    \begin{equation}\label{eq:KNconvlogh}
         \inf_{t\in[0,1]}  (-\log(h(t)))'' -  \frac12 [(-\log(h(t)))']^2 =   \inf_{t\in[0,1]}   \frac 12 \frac{h'(t)^2}{h(t)^2} - \frac{h''(t)}{h(t)} \geq  -  2^{21}\, \delta ^2,
    \end{equation}
    %  \begin{equation*}(-\log(h(t)))'= - \frac{h'(t)}{h(t)} \quad \text{and} \quad (-\log(h(t)))''= - \frac{h''(t)}{h(t)} +  \frac{h'(t)^2}{h(t)^2}, \end{equation*}
    to this aim we divide the problem in two cases.
    First of all, we prove \eqref{eq:KNconvlogh} when $|A-1|\geq 2^{11} \delta$. In this case we have that 
    \begin{equation*}
    \begin{split}
        |h(t) h''(t)| = \big|\phi''(t) \big[ \delta (A-1) + \delta (A-1)^2 (1+&\delta \phi(t))\big]\big| \\
        &\leq 2^7 \left[ \frac{1}{2^{11}} |A-1|^2 + \frac{1}{2^{10}} |A-1|^2 \right] \leq \frac 14 |A-1|^2
    \end{split}
    \end{equation*}
    while, on the other hand,
    \begin{equation*}
        h'(t)^2\geq \frac{1}{2} |A-1|^2.
    \end{equation*}
    Putting together these two inequalities, we conclude that
    \begin{equation*}
        \inf_{t\in[0,1]}  (-\log(h(t)))'' -  \frac12 [(-\log(h(t)))']^2 = \inf_{t\in[0,1]}\frac{h'(t)^2-2 h(t) h''(t)}{2 h(t)^2} \geq 0,
    \end{equation*}
    which in particular implies \eqref{eq:KNconvlogh}.
    Assume now that $|A-1|< 2^{11} \delta $. Notice that
    \begin{equation*}
       \inf_{t\in[0,1]}  (-\log(h(t)))'' -  \frac 12[(-\log(h(t)))']^2 \geq -\sup_{t\in[0,1]} \frac{h''(t)}{h(t)} = -\sup_{t\in[\frac 14,\frac 34]} \frac{h''(t)}{h(t)}
    \end{equation*}
    where the last equality is true because $\phi$ is constant on $\big[0,\frac 14\big] \cup \big[\frac 34, 1\big]$. Moreover, observe that, since $|\delta|<\frac 1{2^{11}}$, we have $h(t)\geq \frac 18$ on $\big[\frac 14, \frac 34\big]$, therefore we deduce
    \begin{equation*}
        \inf_{t\in[0,1]}  (-\log(h(t)))'' -  \frac 12[(-\log(h(t)))']^2 \geq -8 \sup_{t\in[\frac 14,\frac 34]} |h''(t)| \geq -8 \sup_{t\in[0,1]} \delta |\phi''(t)| |A-1|\geq - 2^{21}\, \delta^2,
    \end{equation*} 
    concluding the proof.
\end{pr}

\subsection{Definition of the Metric Measure Spaces}\label{sec:defofMMS}

In this section we define the metric measure spaces that will be studied in the following. 
For every $0<k<\frac14$, introduce the class
\begin{equation}\label{eq:defFk}
    \mathscr{F}_k:= \big\{f \in C^2\big([-1,1]\big) \,:\, 0<f < 3k, \, |f'|\leq k, \, |f''|\leq 1 \big\}.
\end{equation}
Then, for every $f\in \mathscr F_k$ define the set 
\begin{equation*}
    X_f=\{(x,y)\in \setR^2\suchthat x \in [-1,1] \text{ and } 0\leq y\leq f(x)\}.
\end{equation*}
This space will be equipped by the distance $\di_\infty$ induced by the $l_\infty$ norm on $\R^2$, that is
\begin{equation*}
    \di_\infty \big((x_1,y_1), (x_2,y_2) \big) = \max \{ |x_2-x_1|, |y_2-y_1| \}.
\end{equation*}
Observe that, since we imposed $k$ to be less than $\frac 14$, $\di_\infty$ is a geodesic distance on $X_f$. Finally, for every $K\geq1$, define the measure $\m_{f,K}$ on $X_f$ as 
\begin{equation*}
    \m_{f,K} =m_{f,K}(x,y)\cdot\Leb^2|_{X_{f}}:=\frac{1}{f(x)} \exp\left(-K \left(\frac{y}{f(x)} \right)^2\right)\cdot \Leb^2|_{X_f}.
\end{equation*}
\noindent  A simple computation shows that for every $f\in \mathscr F_k$ and $K\geq 1$ it holds that
\begin{equation*}
    (\p_x)_\# \mathfrak{m}_{f,K}=   C_{K}\cdot\chi_{\{-1\leq x\leq 1\}}\cdot\mathcal{H}^1,
\end{equation*}
 where $C_{K}= \int_0^1 e^{-K y^2} \de y$ and $\p_x$ denotes the projection on the $x$-axis.

In section \ref{section:CDproof} we will prove that it is possible to find constants $k\in (0,\frac 14)$, $K\geq1$ and $N>1$ such that, for every $f\in \mathscr F_k$, the metric measure space $(X_f,\di_\infty, \m_{f,K})$ satisfies the $\cd(0,N)$ condition. In the following, we will assume to have fixed an $f\in \mathscr F_k$ and we will develop an argument that only uses the properties of $f$, see \eqref{eq:defFk}, proving in particular the result for the whole class of functions. Moreover, in order to ease the notation, we will usually denote the space $(X_f,\di_\infty, \m_{f,K})$ simply by $(X,\di,\m)$.

\subsection{How to Prove Convexity of the Entropy}\label{sec:jacobi}

In this section we prove an important result (Proposition \ref{prop:jacobi}) that will be a fundamental ingredient in proving the CD condition. The proof of Proposition \ref{prop:jacobi} relies on the possibility to compute the density of a pushforward measure, through Jacobi equation. For example, take two measures $\mu_0,\mu_1\in \ProbTwo(\R^2)$ absolutely continuous with respect to the Lebesgue measure $\Leb^2$, with densities $\rho_0$ and $\rho_1$. If there exists a smooth one-to-one map $T:\R^2 \to \R^2$ such that $T_\#\mu_0=\mu_1$, then Jacobi equation ensures that 
\begin{equation} \label{eq:jacobi}
    \rho_1(T(x,y)) J_T(x,y)= \rho_0(x,y),
\end{equation}
for $\mu_0$-almost every $(x,y)$. The assumptions on the map $T$ can be relaxed in different ways, in this work we are particularly interested in the following version, which is a straightforward consequence of \cite[Proposition 2.1]{rajala2013failure}.

\begin{prop}\label{prop:rajalajacobi}
Let $\mu_0,\mu_1\in \ProbTwo(\R^2)$ be absolutely continuous with respect to the Lebesgue measure $\Leb^2$ and assume that there exists a map $T=(T_1,T_2)$ which is injective outside a $\mu_0$-null set, such that $T_\# \mu_0=\mu_1$. Suppose also that $T_1$ locally does not depend on the $y$ coordinate and it is increasing in $x$, while $T_2$ is increasing in $y$ for every fixed $x$. Then the Jacobi equation \eqref{eq:jacobi} is satisfied with $J_T=\frac{\partial T_1}{\partial x} \frac{\partial T_2}{\partial y}$.
\end{prop}
\noindent This proposition shows in particular that the Jacobi equation can be properly adapted to our setting.

Using Jacobi equation we can deduce the following criterion to prove convexity of the entropy, by looking at local quantities instead of global ones. 

\begin{prop} \label{prop:jacobi}
Let $\mu_0,\mu_1\in \Prob_{ac}(X,\m)$ and $T:X \to X$ be an optimal transport map between $\mu_0$ and $\mu_1$, in particular $T_\# \mu_0 = \mu_1$. Consider a midpoint $\mu_{1/2}\in \Prob_{ac}(X,\m)$ of $\mu_0$ and $\mu_1$, assume that $\mu_{1/2}=[M\circ(\id,T)]_\# \mu_0$ where the map $M:X\times X \to X$ is a (measurable) midpoint selection. Suppose also that the maps $T$ and $M\circ(\id,T):X \to X$ are injective outside a $\mu_0$-null set and they satisfy the Jacobi equation \eqref{eq:jacobi}, with suitable Jacobian functions $J_T$ and $J_{M\circ(\id,T)}$. If 
\begin{equation*}
    \left( m\big(M((x,y),T(x,y))\big) J_{M\circ(\id,T)}(x,y) \right)^\frac 1N \geq \frac12  \left( m(T(x,y)) J_T(x,y) \right)^\frac 1N + \frac12 (m(x,y))^\frac 1N
\end{equation*}
for $\mu_0$-almost every $(x,y)$, then
\begin{equation}\label{eq:entconvexity}
    S_N(\mu_{1/2}) \leq \frac12 S_N(\mu_0) + \frac12 S_N(\mu_1).
\end{equation}
\end{prop}

\begin{pr}

Set $\mu_0=\rho_0 \mathfrak{m}= \tilde{\rho}_0 \Leb^2$, $\mu_1=\rho_1 \mathfrak{m}= \tilde{\rho}_1 \Leb^2$ and $\mu_{1/2}=\rho_{1/2} \mathfrak{m}= \tilde{\rho}_{1/2} \Leb^2$, observe that, in order to prove \eqref{eq:entconvexity}, it is sufficient to prove that 
\begin{equation}\label{eq:sufcond}
     \rho_{1/2}\big(M((x,y),T(x,y))\big)^{-\frac 1N} \geq \frac12  \rho_{1}(T(x,y)) ^{-\frac 1N}  + \frac12 \rho_0(x,y)^{-\frac 1N} ,
\end{equation}
for $\mu_0$-almost every $(x,y)$.
On the other hand, our assumption on the validity of Jacobi equation for $T$ ensures that 
\begin{equation*}
    \tilde{\rho}_1(T(x,y)) J_T(x,y) = \tilde{\rho}_0(x,y),
\end{equation*}
for $\mu_0$-almost every $(x,y)$, and thus that 
\begin{equation*}
    \rho_1(T(x,y)) m(T(x,y)) J_T(x,y) = \rho_0(x,y) m(x,y).
\end{equation*}
for $\mu_0$-almost every $(x,y)$.
Analogously, since the Jacobi equation holds for $M\circ(\id,T)$, we can deduce that
\begin{equation*}
    \rho_{1/2}\big(M((x,y),T(x,y))\big) m\big(M((x,y),T(x,y))\big) J_{M\circ(\id,T)}(x,y) = \rho_0(x,y) m(x,y),
\end{equation*}
for $\mu_0$-almost every $(x,y)$.
Therefore, \eqref{eq:sufcond} is equivalent to 
\begin{equation*}
     \left( \frac{\rho_0(x,y) m(x,y)}{ m\big(M((x,y),T(x,y))\big) J_{M\circ(\id,T)}(x,y)}\right)^{-\frac 1N} \geq 
    \frac12  \left( \frac{\rho_0(x,y) m(x,y)}{m(T(x,y)) J_T(x,y)} \right)^{-\frac 1N} + \frac12 (\rho_0(x,y))^{-\frac 1N}.
\end{equation*}
Some easy rearrangements show that this last equation is equivalent to
\begin{equation*}
    \left( m\big(M((x,y),T(x,y))\big) J_{M\circ(\id,T)}(x,y) \right)^\frac 1N \geq \frac12  \left( m(T(x,y)) J_T(x,y) \right)^\frac 1N + \frac12 (m(x,y))^\frac 1N,
\end{equation*}
concluding the proof.
\end{pr}

\noindent In the proof of the main result (Theorem \ref{thm:CDTheExample}) we use Proposition \ref{prop:jacobi}, together with Proposition \ref{prop:CDinmidpoint}, to prove the $\cd(0,N)$ condition. This will is possible because the maps $T$ and $M\circ(\id,T)$, that we are going to consider, satisfy the assumptions of Proposition \ref{prop:rajalajacobi}.

\subsection{Definition of the Midpoint}\label{sec:defimidpoint}

According to Proposition \ref{prop:CDinmidpoint}, in order to prove $\cd(0,N)$ condition, it is sufficient to show entropy convexity in a suitable midpoint of any pair of absolutely continuous measures. Observe that in highly branching metric measure spaces, like $(X,\di,\m)=(X_f,\di_\infty, \m_{f,K})$, the choice of a midpoint can be done with great freedom. This is because, in general, both the optimal transport map and the geodesic interpolation are not unique. In this section, we present the midpoint selection used in \cite{MR4402722}.
In particular, for any pair $\mu_0,\mu_1 \in \Prob_{ac}(X,\m)$, we select a suitable optimal transport map $T$ between them and then we identify a Wasserstein midpoint with a suitable midpoint interpolation map $M$. 
In order to define both the optimal transport map $T$ and the midpoint interpolation map $M$, we introduce the sets $V,D,H,H_0,H_\frac12,H_1 \subset \R^2 \times \R^2$ as:
\begin{equation*}
    V:=\left\{\left(\left(x_{0}, y_{0}\right),\left(x_{1}, y_{1}\right)\right) \in \mathbb{R}^{2} \times \mathbb{R}^{2}:\left|x_{0}-x_{1}\right|<\left|y_{0}-y_{1}\right|\right\},
\end{equation*}
\begin{equation*}
    D:=\left\{\left(\left(x_{0}, y_{0}\right),\left(x_{1}, y_{1}\right)\right) \in \mathbb{R}^{2} \times \mathbb{R}^{2}:\left|x_{0}-x_{1}\right|=\left|y_{0}-y_{1}\right|\right\},
\end{equation*}
\begin{equation*}
    H:=\left\{\left(\left(x_{0}, y_{0}\right),\left(x_{1}, y_{1}\right)\right) \in \mathbb{R}^{2} \times \mathbb{R}^{2}:\left|x_{0}-x_{1}\right|>\left|y_{0}-y_{1}\right|\right\}= H_0  \cup H_1,
\end{equation*}
where
\begin{equation*}
    H_0:=\left\{\left(\left(x_{0}, y_{0}\right),\left(x_{1}, y_{1}\right)\right) \in \mathbb{R}^{2} \times \mathbb{R}^{2}:\frac12\left|x_{0}-x_{1}\right|\geq\left|y_{0}-y_{1}\right|\right\},
\end{equation*}
\begin{equation*}
    H_1:=\left\{\left(\left(x_{0}, y_{0}\right),\left(x_{1}, y_{1}\right)\right) \in \mathbb{R}^{2} \times \mathbb{R}^{2}:\left|x_{0}-x_{1}\right|>\left|y_{0}-y_{1}\right| >\frac12 |x_0-x_1|\right\}.
\end{equation*}

First of all we present our optimal transport map selection, which follows the work of Rajala \cite{rajala2013failure}. In particular, given two absolutely continuous measures $\mu_0,\mu_1 \in \Prob(\R^2)$, he was able to select, via consecutive minimizations, an optimal transport map with different nice properties, which are summarized in the following statement.

\begin{prop}\label{prop:map}
Given two measures $\mu_0,\mu_1 \in \Prob(\R^2)$ which are absolutely continuous with respect to the Lebesgue measure $\Leb^2$, there exists a measurable optimal transport map $T=(T_1,T_2)$ between $\mu_0$ and $\mu_1$, injective outside a $\mu_0$-null set, with the following properties. For $\mu_0$-almost every $(x,y)$, we have that
\begin{equation*}
    \begin{array}{l}T_{1} \text { is locally constant in } y, \text { if }((x, y), T(x, y)) \in H \text { and } \\ T_{2} \text { is locally constant in } x, \text { if }((x, y), T(x, y)) \in V\end{array}.
\end{equation*}
Moreover, the function $T_1(x,y)$ is increasing in $x$ for every fixed $y$ and the function $T_2(x,y)$ is increasing in $y$ for every fixed $x$, therefore for $\mu_0$-almost every $(x,y)$ it holds
\begin{equation*}
    \frac{\partial T_{1}}{\partial x} \geq 0 \text { and } \frac{\partial T_{2}}{\partial y} \geq 0, \text { if }((x, y), T(x, y)) \in H \cup V.
\end{equation*}
\end{prop}

Now fix two measures $\mu_0,\mu_1\in \Prob_{ac}(X,\m)$, observe that, since they are absolutely continuous with respect to the reference measure $\m$, they are absolutely continuous also with respect to the Lebesgue measure $\Leb^2$. Call $T$ the optimal transport map between $\mu_0$ and $\mu_1$, identified by Proposition \ref{prop:map}. In order to identify a midpoint of $\mu_0$ and $\mu_1$, we need to select a proper midpoint interpolation map, i.e. a measurable map $M:X\times X\to X$ such that 
\begin{equation*}
    \di_\infty(M(z,w),z) = \di_\infty(M(z,w),w)= \frac12 \di_\infty(z,w) \quad \text{for every }(z,w)\in X \times X ,
\end{equation*}
the desired midpoint will be $M_\# \big( (\id,T)_\# \mu_0 \big)= [M\circ (\id,T)]_\# \mu_0$.

The midpoint interpolation map $M$ that we are going to use in the following is defined in different ways on the sets $V$, $D$, $H_0$ and $H_1$. In particular, the precise definition is the following:
\begin{itemize}
    \item If $\big((x_0,y_0), (x_1,y_1)\big)\in V\cup D$
    \begin{equation*}
        M\big((x_0,y_0), (x_1,y_1)\big):=\left(\frac{x_0+x_1}{2}, \frac{y_0+y_1}{2} \right)
    \end{equation*}
    \item If $\big((x_0,y_0), (x_1,y_1)\big)\in H_0 $,
    \begin{equation*}
    M\big((x_0,y_0), (x_1,y_1)\big)= \left(\frac{x_0+x_1}{2} , \frac12 \bigg( \frac{y_0}{f(x_0)}  +  \frac{y_1}{f(x_1)}  \bigg)f\left(\frac{x_0+x_1}{2}\right) \right).\end{equation*}
    \item If $\big((x_0,y_0), (x_1,y_1)\big)\in H_1$, with $x_0<x_1$ and $y_0<y_1$, introduce the quantity
    \begin{equation*}
        \tilde{y}(x_0,x_1,y_0)= \frac12\bigg( \frac{y_0}{f(x)}  +  \frac{y_0+\frac{x_1-x_0}{2}}{f(x_1)}  \bigg)f\left(\frac{x_1+x_0}{2}\right)  -y_0,
    \end{equation*}
    and consequently define 
    \begin{align*}
        M&\big((x_0,y_0), (x_1,y_1)\big)\\
        & \qquad\qquad= \left(\frac{x_0+x_1}{2} ,y_0 + \tilde{y}(x_0,x_1,y_0) + \left(\frac{x_1-x_0}{2} - \tilde{y}(x_0,x_1,y_0) \right) \left( 2 \frac{y_1 - y_0}{x_1-x_0} -1 \right)\right).
    \end{align*}
    In the other cases where $\big((x_0,y_0), (x_1,y_1)\big)\in H_1$, $M$ can be defined analogously, every proof from now on will be done only taking care of this case, implying it can be easily adapted to the other cases.
\end{itemize}

The next statement, which combines Proposition 5.3 and Proposition 5.4 in \cite{MR4402722}, ensures that $M$ actually provides a midpoint selection and interacts well with the selected optimal transport map $T$.

\begin{prop}\label{prop:aeinjectivity}
For $k$ sufficiently small, the map $M$ is a midpoint interpolation map and the map $M\circ(\id,T)$ is injective outside a $\mu_0$-null set.
\end{prop}

\begin{remark}
    The proof Proposition 5.3 and Proposition 5.4 in \cite{MR4402722} is done for a quite specific choice of the function $f$ which is defining the profile of $X=X_f$. However the proof only requires $f$ to satisfy the properties which define the set $\mathscr{F}_k$ (see \eqref{eq:defFk}), therefore the statement is true for every $f\in \mathscr{F}_k$.
\end{remark}

\section{Proof of CD Condition}\label{section:CDproof}

In this section we prove the main result of this work, showing the validity of the $\CD(0,N)$ condition for metric measure spaces of the type $(X_f,\di_\infty, \m_{f,K})$, with $f\in \mathscr{F}_k$. The proof follows the strategy developed in \cite{MR4402722}, refining it in order to obtain the $\cd$ condition with finite dimensional parameter.

\begin{theorem}\label{thm:CDTheExample}
For $k$ sufficiently small and $K$ sufficiently large, there exists $N>1$ such that for every $f\in \mathscr{F}_k$ the metric measure space $(X_f,\di_\infty, \m_{f,K})$ is a $\CD(0,N)$ space.
\end{theorem}

Before going into the proof of Theorem \ref{thm:CDTheExample}, we need the following preliminary lemma, which is an improvement of \cite[Lemma 3.3]{MR4402722}.

\begin{lemma}\label{lem:est3}
Having fixed the constant $K\geq 1$ and given another constant $H>0$, it is possible to find $k$ sufficiently small such that the following holds. Given any $C^2$ function $y:I=[x_0,x_1] \to \setR^+$ such that $y'(x)\geq \frac14$ and $y''(x)\leq H \frac{k}{f(x)}$ for every $x \in I$, and calling $f_I$ the maximum of $f$ on the interval $I=[x_0,x_1]$ (i.e. $f_I=\max_{x_0\leq x\leq x_1} f(x)$), the function $-\log(m(x,y(x)))$ is $(\frac{K}{32f_I^2}, 32 K)$-convex.
\end{lemma}

\begin{pr}
From the proof of \cite[Lemma 3.3]{MR4402722}, we know that 
\begin{equation*}
      \frac{\partial}{\partial x} \big(-\log(m(x,y(x)))\big) = \frac{f'(x)}{f(x)} +2K \frac{y(x)}{f(x)} \left( \frac{y'(x)}{f(x)} - \frac{y(x)f'(x)}{f(x)^2} \right),
\end{equation*}
in particular, keeping in mind that $\modu{\frac{y(x)}{f(x)}}\leq 1$ and $\modu{f'(x)}\leq k$, we deduce that
\begin{equation*}
\begin{split}
    \bigg|\frac{\partial}{\partial x} \big(-\log(m(x,y(x)))\big)\bigg| &\leq \bigg| \frac{f'(x)}{f(x)} \bigg| + 2K\bigg| \frac{y(x) y'(x)}{f(x)^2}\bigg| + 2K\bigg|\frac{y(x)f'(x)}{f(x)^3}\bigg|\\
    &\leq \frac{k}{f(x)} + 2K \frac{y'(x)}{f(x)} + 2K \frac{k}{f(x)} \leq 4K \frac{y'(x)}{f(x)},
\end{split}
\end{equation*}
where the last inequality holds for $k$ sufficiently small.
Moreover, from the computations done in the proof of \cite[Lemma 3.3]{MR4402722}, we deduce that for $k$ sufficiently small 
\begin{align*}
    \frac{\partial^2}{\partial x^2} \big(-\log(m(x,y(x)))\big) \geq \frac{K}{ f(x)^2}\, y'(x)^2.
\end{align*}
 We can then conclude that
 \begin{equation*}
    \begin{split}
         \frac{\partial^2}{\partial x^2} \big(-\log(m(x,y(x)))\big) \geq \frac{K}{ f(x)^2}\, y'(x)^2  &\geq  \frac{K}{32 f_I^2} + \frac{1}{32 K} \left[ 4K \frac{y'(x)}{f(x)}\right]^2 \\
         &\geq  \frac{K}{32 f_I^2} + \frac{1}{32 K} \left[\frac{\partial}{\partial x} \big(-\log(m(x,y(x)))\big)\right]^2,
    \end{split}
 \end{equation*}
 which proves the $(\frac{K}{32f_I^2}, 32 K)$-convexity.
\end{pr}

\begin{proof}[Proof of Theorem \ref{thm:CDTheExample}]
Fix $f\in \mathscr{F}_k$ and consider the metric measure space $(X,\di,\m)=(X_f,\di_\infty, \m_{f,K})$. We are going to prove that, for $k$ sufficiently small and $K$ sufficiently large, $(X,\di,\m)$ satisfies the $\CD(0,N)$ condition, for a suitable $N>1$. Moreover, the whole argument will not depend on the specific choice of $f$, but only on the properties defining the set $\mathscr{F}_k$ (which $f$ satisfies). In particular, the choice of the parameters $k$, $K$ and $N$ will be independent on $f\in \mathscr{F}_k$ and this will be sufficient to prove the statement.

Let $\mu_0,\mu_1\in \Prob_{ac}(X,\m)$, then, according to Proposition \ref{prop:CDinmidpoint}, it is sufficient to prove that, for every $N'\geq N$, we have 
\begin{equation*}
    S_{N'}(\mu_{1/2}) \leq (1-t) S_{N'}(\mu_0) + t S_{N'}(\mu_1),
\end{equation*}
where $\mu_{1/2}=[M\circ (\id,T)]_\# \mu_0$ is the midpoint selected in Section \ref{sec:defimidpoint}.
Given Proposition \ref{prop:jacobi}, it is enough check that for every $N'\geq N$
\begin{equation}\label{eq:condition}
      \left( m\big(M((x,y),T(x,y))\big) J_{M\circ(\id,T)}(x,y) \right)^\frac 1{N'} \geq \frac12  \left( m(T(x,y)) J_T(x,y) \right)^\frac 1{N'} + \frac12 (m(x,y))^\frac 1{N'},
\end{equation}
for $\mu_0$-almost every $(x,y)$. We are going to prove \eqref{eq:condition}, in different cases depending on the pair $((x,y),T(x,y))$. Observe that, since we have to prove \eqref{eq:condition} for $\mu_0$-almost every $(x,y)$, we can assume that the conclusions of Proposition \ref{prop:map} hold for every $(x,y)$ we are considering.

Let us start by proving \eqref{eq:condition} for the points $(x,y)$ such that $((x,y),T(x,y))\in H_0 $, recall that in this case we have
\begin{equation*}
    M\circ(\id,T) (x,y)=  \left(\frac{x+T_1}{2} , \frac12 \bigg( \frac{y}{f(x)}  +  \frac{T_2}{f(T_1)}  \bigg)f\left(\frac{x+T_1}{2}\right) \right).
\end{equation*}
In particular, keeping in mind Proposition \ref{prop:aeinjectivity} we can easily see that $M\circ(\id,T)$ satisfies the assumption of Proposition \ref{prop:rajalajacobi}, therefore it holds that
\begin{equation*}
    J_{M\circ(\id,T)} (x,y) = \frac{1}{2} \left( 1 + \frac{\partial T_1}{\partial x} \right) \frac{1}{2} \bigg( \frac{1}{f(x)}  +  \frac{\frac{\partial T_2}{\partial y}}{f(T_1)}  \bigg) f\left( \frac{x+T_1}{2} \right).
\end{equation*}
Moreover, we have that 
\begin{equation*}
    m\big(M((x,y),T(x,y))\big) = f\left( \frac{x+T_1}{2} \right)^{-1} \exp\left( \frac{-K}{4} \bigg( \frac{y}{f(x)}  +  \frac{T_2}{f(T_1)}  \bigg)^2\right),
\end{equation*}
thus, putting together these last two equation, we deduce that
\begin{equation*}
\begin{split}
    - \log \Big( m&\big(M((x,y),T(x,y))\big) J_{M\circ(\id,T)}(x,y) \Big) \\
    &= -\log \Bigg( \frac{1}{2} \left( 1 + \frac{\partial T_1}{\partial x} \right) \frac{1}{2} \bigg( \frac{1}{f(x)}  +  \frac{\frac{\partial T_2}{\partial y}}{f(T_1)}  \bigg)  
    \exp\Bigg( \frac{-K}{4} \bigg( \frac{y}{f(x)}  +  \frac{T_2}{f(T_1)}  \bigg)^2\Bigg) \Bigg)\\
    &= -\log \left(\frac{1}{2} \left( 1 + \frac{\partial T_1}{\partial x} \right) \right) - \log \left( \frac{1}{2} \bigg( \frac{1}{f(x)}  +  \frac{\frac{\partial T_2}{\partial y}}{f(T_1)}  \bigg)\right) + K  \bigg(\frac12 \bigg( \frac{y}{f(x)}  +  \frac{T_2}{f(T_1)}  \bigg) \bigg)^2.
\end{split}
\end{equation*}
On the other hand, we have that
\begin{equation*}
    -\log(m(x,y)) = -\log\left(\frac{1}{f(x)} \exp \left(-K \bigg(\frac{y}{f(x)}\bigg)^2 \right) \right)=- \log(1)- \log\left(\frac{1}{f(x)} \right) +K \bigg(\frac{y}{f(x)}\bigg)^2 
\end{equation*}
and, applying once again Proposition \ref{prop:rajalajacobi}, this time to the map $T$, 
\begin{align*}
    -\log \left( m(T(x,y)) J_T(x,y) \right)&= -\log \left( \frac{\partial T_1}{\partial x} \frac{\partial T_2}{\partial y} \frac{1}{f(T_1)} \exp \left(-K \bigg(\frac{T_2}{f(T_1)}\bigg)^2 \right)\right)\\
    &=-\log \left( \frac{\partial T_1}{\partial x} \right)- \log \left( \frac{\frac{\partial T_2}{\partial y}}{f(T_1)}\right)+K \bigg(\frac{T_2}{f(T_1)}\bigg)^2.
\end{align*}
Observe now that, combining the statements of Example \ref{ex:KNconv}, Remark \ref{rmk:repar} and Lemma \ref{lem:sumKN}, we deduce that the function 
\begin{equation*}
    [0,1]\ni t \mapsto -\log \left(  (1-t) + t \frac{\partial T_1}{\partial x}  \right) - \log \left(  (1-t) \frac{1}{f(x)}  +  t \frac{\frac{\partial T_2}{\partial y}}{f(T_1)}  \right) + K  \bigg( (1-t)\frac{y}{f(x)}  +  t \frac{T_2}{f(T_1)} \bigg)^2.
\end{equation*}
is $(0,2K+2)$-convex and therefore $(0,N')$-convex for every $N'\geq 2K+2$. Then, according to \eqref{eq:0Nconv=concavity}, we deduce that \eqref{eq:condition} holds for every $N'\geq 2K+2$.

We now proceed to prove \eqref{eq:condition} for ($\mu_0$-almost) every $(x,y)$ such that $((x,y),T(x,y))\in V$, the strategy is similar but requires Lemma \ref{lem:est3}. In this case the midpoint interpolation map is trivial, therefore it holds that 
\begin{equation*}
    J_{M\circ(\id,T)} (x,y) = \frac{1}{2} \left( 1 + \frac{\partial T_1}{\partial x} \right) \cdot  \frac{1}{2} \left( 1 + \frac{\partial T_2}{\partial y} \right).
\end{equation*}
\\
In particular, we have that 
\begin{equation*}
\begin{split}
    - \log \Big( m&\big(M((x,y),T(x,y))\big) J_{M\circ(\id,T)}(x,y) \Big) \\
    &= -\log \left(\frac{1}{2} \left( 1 + \frac{\partial T_1}{\partial x} \right) \right) - \log \left(\frac{1}{2} \left( 1 + \frac{\partial T_2}{\partial y} \right) \right) - \log \Big( m\big(M((x,y),T(x,y))\big)  \Big)
\end{split}
\end{equation*}
and, on the other hand,
\begin{equation}\label{eq:jacobiendpoint}
     -\log \left( m(T(x,y)) J_T(x,y) \right)=-\log \left( \frac{\partial T_1}{\partial x} \right)- \log \left(\frac{\partial T_2}{\partial y}\right)-\log \left( m(T(x,y)) \right).
\end{equation}
Now, we can assume without loss of generality that $x<T_1$ and define the function $z:[x,T_1]\to \R$ parameterizing the segment connecting $(x,y)$ and $T(x,y)$, i.e. $z(t)= y + (T_2-y)\frac{t-x}{T_1-x}$ for every $t\in[x,T_1]$. Then, applying Lemma \ref{lem:est3}, we deduce that the function $-\log(m(t,z(t)))$ is $(\frac{K}{32f_{[x,T_1]}^2}, 32 K)$-convex (on $[x,T_1]$), thus also $(0, 32 K)$-convex. (Observe that when $x>T_1$ we can use an analogous argument, while, when $x=T_1$ the direct computation yields the same convexity.) Then, we can proceed as in the previous case and conclude that \eqref{eq:condition} holds for every $N'\geq 32K+2$. The case of $(x,y)$ such that $((x,y),T(x,y))\in D$ can be solved following the same strategy, after a change of variable, as done in \cite{rajala2013failure}.

We are left with the last case, which consists in proving \eqref{eq:condition} for ($\mu_0$-almost) every $(x,y)$ such that $((x,y),T(x,y))\in H_1$ (with $x<T_1(x,y)$ and $y<T_2(x,y)$). Before developing the argument, we notice that 
\begin{equation*}
    f(T_1) \geq T_2 \geq y + \frac{T_1-x}{2} \geq \frac{T_1-x}{2},
\end{equation*}
therefore, since $|f'|\leq k$, we deduce that
\begin{equation}\label{eq:stimafI}
  f_I:=f_{[x,T_1]}= \max_{x\leq r\leq T_1} f(r)  \leq f(T_1) + k (T_1-x) \leq (1+2k) f(T_1) \leq 2 f(T_1),
  \end{equation}
  for $k$ sufficiently small. Back to the argument, consider the map 
\begin{align*}
    (S_1,S_2)(x,y)&:= M \circ (\id, T)(x,y) \\
   &\quad = \left(\frac{x+T_1}{2}, y + \tilde{y}(x,T_1,y) + \left(\frac{T_1-x}{2} - \tilde{y}(x,T_1,y) \right) \left( 2 \frac{T_2 - y}{T_1-x} -1 \right) \right).
\end{align*}
Proceeding as before, Proposition \ref{prop:rajalajacobi} and Proposition \ref{prop:aeinjectivity} ensure that
\begin{equation*}
     J_{M\circ(\id,T)} (x,y)= \frac{\partial S_1}{\partial x} \frac{\partial S_2}{\partial y},
\end{equation*}
and therefore we have
\begin{equation*}
\begin{split}
    - \log \Big( m\big(M((x,y),T(x,y))\big)& J_{M\circ(\id,T)}(x,y) \Big) \\
    &= -\log \left( \frac{\partial S_1}{\partial x}\right) -\log \left( \frac{\partial S_2}{\partial y}\right) - \log \Big( m\big(M((x,y),T(x,y))\big)  \Big),
\end{split}
\end{equation*}
while, as before, \eqref{eq:jacobiendpoint} holds.
On the one hand, we easily have that    
\begin{equation*}
    \frac{\partial S_1}{\partial x}=\frac{1}{2} \left( 1 + \frac{\partial T_1}{\partial x} \right),
\end{equation*}
and, on the other hand, from the proof of Theorem 6.1 in \cite{MR4402722} we know that 
\begin{align*}
    \frac{\partial S_2}{\partial y} = 1+ \frac{\partial}{\partial y}\tilde{y}(x,T_1,y) \left( 2 - 2 \frac{T_2 - y}{T_1-x} \right) + \left( \frac{\partial T_2}{\partial y}-1 \right) \left(\frac12 - \frac{\tilde{y}(x,T_1,y)-\frac{T_1-x}{4}}{\frac{T_1-x}{2}} \right).
\end{align*}
Moreover, the computations in the proof of Theorem 6.1 in \cite{MR4402722} show that 
\begin{equation*}
    \modu{\frac{\tilde{y}(x,T_1,y)-\frac{T_1-x}{4}}{\frac{T_1-x}{2}}} \leq \frac{[2k^2 + 4k]\frac{T_1-x}{2}}{f(T_1)} < \frac{1}{2^{12}} \frac{\frac{T_1-x}{2}}{f(T_1)} \leq \frac{1}{2^{11}} \frac{\frac{T_1-x}{2}}{f_I},
\end{equation*}
where the second inequality holds for a sufficiently small $k$ and the third follows from \eqref{eq:stimafI}.
Now, suppose that 
\begin{equation*}
    \frac{\partial}{\partial y}\tilde{y}(x,T_1,y)=  \frac 12 \frac{f\big(\frac{x+T_1}{2}\big)}{f(x)} + \frac 12 \frac{f\big(\frac{x+T_1}{2}\big)}{f(T_1)}-1 \geq0 .
\end{equation*}
Then, after noticing that $\frac{T_1-x}{2}\leq f(T_1)\leq f_I$, we can apply Lemma \ref{lem:approxconv} and find a function $h:[0,1]\to \R$, with $h(0)=1$, $h(1)= \frac{\partial T_2}{\partial y}$ and 
\begin{equation}\label{eq:excess}
    h(1/2)= 1 + \left( \frac{\partial T_2}{\partial y}-1 \right) \left(\frac12 - \frac{\tilde{y}(x,T_1,y)-\frac{T_1-x}{4}}{\frac{T_1-x}{2}} \right) \leq \frac{\partial S_2}{\partial y}
\end{equation}
such that 
\begin{equation}\label{eq:convpazza}
    -\log(h(t)) \text{ is }\bigg(-  2^{21}\, \bigg[\frac{T_1-x}{2^{11} f_I}\bigg]^2,2\bigg)= \bigg(- \frac 12 \, \bigg[\frac{T_1-x}{f_I}\bigg]^2,2\bigg) \text{-convex}.
\end{equation}
On the other hand, following the argument in the proof of Theorem 6.1 in \cite{MR4402722}, we can find a function $z:[x,T_1]\to R$ satisfying the assumption of Lemma \ref{lem:est3} such that 
\begin{equation*}
    \bigg(\frac{x+T_1}{2},z\bigg(\frac{x+T_1}{2}\bigg)\bigg)= M((x,y),T(x,y))
\end{equation*}
In particular, according to Lemma \ref{lem:est3} and Remark \ref{rmk:repar} the function 
\begin{equation*}
    t \mapsto - \log\Big(m\Big((1-t)x+ tT_1, z\big((1-t)x+ tT_1\big)\Big)\Big)
\end{equation*}
is $(\frac{K}{32f_I^2} (T_1-x)^2, 32 K)$-convex. Then, according to the Lemma \ref{lem:sumKN} and keeping in mind \eqref{eq:convpazza}, whenever $K$ is sufficiently large (i.e. $K\geq 16$) we have that the function 
\begin{equation*}
    t \mapsto -\log\left( 1 + t\left(\frac{\partial T_1}{\partial x}-1\right) \right)   -\log(h(t)) - \log\Big(m\Big((1-t)x+ tT_1, z\big((1-t)x+ tT_1\big)\Big)\Big) 
\end{equation*}
is $(0,32K + 3)$-convex. Proceeding as in the first case and keeping in mind the inequality in \eqref{eq:excess}, we prove that \eqref{eq:condition} holds for every $N'\geq32K+3$. 

If instead 
\begin{equation*}
    \frac{\partial}{\partial y}\tilde{y}(x,T_1,y)=  \frac 12 \frac{f\big(\frac{x+T_1}{2}\big)}{f(x)} + \frac 12 \frac{f\big(\frac{x+T_1}{2}\big)}{f(T_1)}-1 <0 ,
\end{equation*}
the argument can be adapted following the same strategy developed in the proof of Theorem 6.1 in \cite{MR4402722}. We are then able to prove \eqref{eq:condition} for $\mu_0$-almost every $(x,y)$ in every case, concluding the proof.
\end{proof}

Now, we want to combine Theorem \ref{thm:CDTheExample} with the stability of the $\cd(0,N)$ condition with respect to the measured Gromov-Hausdorff convergence, in order to prove the $\cd(0,N)$ condition for singular spaces. To this aim, we introduce the set 
\begin{equation*}
    \overbar{\mathscr{ F}_k}:= \big\{f \in C^2\big([-1,1]\big) \,:\, 0\leq f < 3k, \, |f'|\leq k, \, |f''|\leq 1 \big\} \supset \mathscr F_k,
\end{equation*}
and we extend the definition of the metric measure space $(X_f,\di_\infty, \m_{f,K})$ to all functions $f \in \overbar{\mathscr{ F}_k}$, in analogy to what did in Section \ref{sec:defofMMS}. In particular, given any $\overbar{\mathscr{ F}_k}$, the definition of the space $X_f$ is the same as before, i.e.
\begin{equation*}
    X_f=\{(x,y)\in \setR^2\suchthat x \in [-1,1] \text{ and } 0\leq y\leq f(x)\}.
\end{equation*}
while the measure $\m_{f,K}$ becomes singular:  \begin{equation*}
    \m_{f,K} := \mathbbm{1}_{\{f(x)=0\}}\cdot  C_K \cdot\mathcal   H^1 |_{y=0} + \mathbbm{1}_{\{f(x)>0\}} \cdot\frac{1}{f(x)} \exp\left(-K \left(\frac{y}{f(x)} \right)^2\right)\cdot \Leb^2|_{X_f},
\end{equation*}
where, as before, $C_{K}= \int_0^1 e^{-K y^2} \de y$.

\begin{remark}\label{rmk:mGH}
    The measured Gromov-Hausdorff convergence is a notion of convergence for metric measure spaces that basically combines the Hausdorff convergence for the metric side and the weak convergence for the reference measures. It has different equivalent definitions (see \cite{Gigli_2015}), but for the purpose of this paper it is sufficient to consider the definition given in Villani's book \cite[Definition 27.30]{villani2008}. The $\cd$ condition is stable with respect to the measure Gromov-Hausdorff convergence, see Theorem 29.24 and Theorem 29.25 in \cite{villani2008}.
\end{remark}

\begin{corollary}\label{cor:main}
    For $k$ sufficiently small and $K$ sufficiently large, there exists $N>1$ such that for every $f\in \overbar{\mathscr{F}_k}$ the metric measure space $(X_f,\di_\infty, \m_{f,K})$ is a $\CD(0,N)$ space.
\end{corollary}

\begin{proof}
    For suitable $k$ and $K$, Theorem \ref{thm:CDTheExample} guarantees the existence of $N>1$ such that $(X_g,\di_\infty, \m_{g,K})$ is a $\cd(0,N)$ space for every $g\in \mathscr{F}_k$. Now, given $f\in \overbar{\mathscr{F}_k}$, take a sequence $(\varepsilon_n)_{n\in \mathbb N}$ of positive numbers converging to $0$, notice that definitely $f+\varepsilon_n\in \mathscr{F}_k$ and thus $(X_{f+\varepsilon_n},\di_\infty, \m_{f+\varepsilon_n,K})$ is a $\CD(0,N)$ space. Moreover, it is easy to realize that 
    \begin{equation}\label{eq:mGHconv}
        (X_{f+\varepsilon_n},\di_\infty, \m_{f+\varepsilon_n,K}) \longrightarrow (X_f,\di_\infty, \m_{f,K}) \quad\text{as }n\to \infty,
    \end{equation}
    with respect to the the measured Gromov-Hausdorff convergence. From the stability of the $\cd(0,N)$ condition (see Remark \ref{rmk:mGH}), we conclude that $(X_f,\di_\infty, \m_{f,K})$ is a $\CD(0,N)$ space.
\end{proof}

\begin{remark}
    For a formal proof of the measured Gromov-Hausdorff convergence \eqref{eq:mGHconv} we refer to the proof of Theorem 7.1 in \cite{MR4402722}, the setting of \cite{MR4402722} is less general but the strategy developed in that case works also in the context of this work.
\end{remark}

\begin{remark}
    Observe that, the fact that the constant $N>1$ in Theorem \ref{thm:CDTheExample} does not depend on the specific choice of the function $f\in \mathscr F_k$, is crucial in the proof of Corollary \ref{cor:main}.
\end{remark}

\section{Conclusions}

\subsection{Examples of singular $\cd(0,N)$ spaces}\label{sec:conc1}

Using Corollary \ref{cor:main} we can construct interesting examples of singular $\cd(0,N)$ spaces, in fact, given a function $f\in \overbar{\mathscr{F}_k}\setminus\mathscr{F}_k$, the measured Gromov-Hausdorff limit procedure in the proof of Corollary \ref{cor:main} makes the $y$-dimension collapse in $\{f(x)=0\}$. For example, taking $f\in \overbar{\mathscr{F}_k}$ increasing and such that $\{f=0\}=[-1,0]$, the metric measure space $(X_f,\di_\infty, \m_{f,K})$, which is $\cd(0,N)$, is similar to the representation in Figure \ref{fig:example1}.
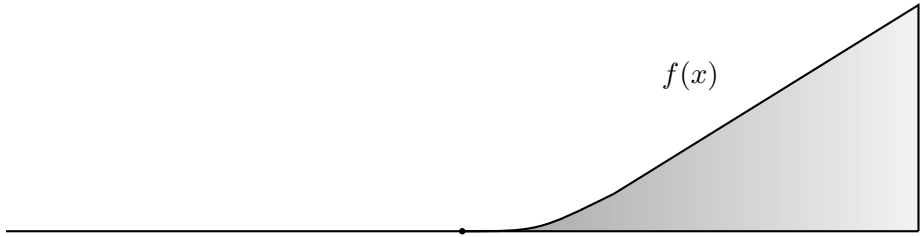
\begin{figure}[h]
\begin{center}

\begin{tikzpicture} 

\shade[left color=black!37!white,right color=black!5!white] (0,0)--(6,0)--(6,3)--(2,0.5)..controls (1,0)..(0,0)--cycle;
\draw[thick](-6,0)--(6,0);
\filldraw[black] (0,0) circle (1pt);
\draw[thick](6,0)-- (6,3)--(2,0.5);
\draw[thick](0,0)..controls (1,0)..(2,0.5); 
%\filldraw[white] (-7,0) circle (1pt);
\node at (3,1.6)[label=north:$f(x)$] {};
\end{tikzpicture}

\caption{The metric measure space $(X_f,\di_\infty, \m_{f,K})$.}
\label{fig:example1}
\end{center}
\end{figure}

Observe that this space has different topological and Hausdorff dimensions in different regions, in fact $X_f \cap ([-1,0]\times \R)$ has (topological and Hausdorff) dimension 1, while $X_f \cap ([0,1]\times \R)$ has (topological and Hausdorff) dimension 2. This observation generalizes a result by Ketterer and Rajala \cite{ketterer2014failure}, who constructed a space with non-constant dimension, satisfying the so-called measure contraction property $\MCP$ (see \cite{OhtaMCP} for the definition). The extension of this result to $\cd(0,N)$ spaces is not obvious, since the $\cd$ condition is strictly stronger than the measure contraction property and forces the space to be a little bit more rigid. 

Moreover, as highlighted in the introduction, the proved possible non-constancy of the dimension for $\cd(0,N)$ spaces is especially interesting in relation to what happens in the context of $\RCD$ spaces, that are CD spaces which are also infinitesimally Hilbertian (cfr. \cite{giglidiff,Ambrosio_2014}). In fact, it was proved by Brué and Semola in \cite{bruesemola} that every $\RCD(K,N)$ space has constant dimension. The example presented in this section proves that the same is not true for $\CD(0,N)$ spaces, showing in particular that, as expected, the infinitesimal Hilbertianity assumption is necessary in \cite{bruesemola}. 

The function $f\in \overbar{\mathscr{F}_k}\setminus\mathscr{F}_k$ we have considered up to now in this section identifies a set $X_f$, which is basically the same as the one considered in \cite{MR4402722}. However, the more general approach adopted in this work allows to identify other examples of $\cd$ spaces with different shapes. In particular, Corollary \ref{cor:main} proves that the $\cd(0,N)$ condition can be verified also by spaces having the shapes represented in Figure \ref{fig:example2}. This shows that one-dimensional and two-dimensional parts can be alternated along $\{y=0\}$ and the $\cd(0,N)$ can still be true.

\begin{figure}[h]
\begin{center}

\begin{tikzpicture} 

\filldraw[white] (0,2.8) circle (1pt);
\shade[left color=black!40!white,right color=black!5!white] (-7,0)..controls (-6,0) and (-6,2) ..(-5,2)-- (-5,0)--cycle;
\shade[left color=black!5!white,right color=black!40!white] (-3,0)..controls (-4,0) and (-4,2) ..(-5,2)-- (-5,0)--cycle;
\draw[thick](-8,0)--(-2,0);
\draw[thick] (-7,0)..controls (-6,0) and (-6,2) ..(-5,2)..controls (-4,2) and (-4,0)..(-3,0);

\shade[left color=black!5!white,right color=black!40!white](0,0)--(0,2)..controls (2,0).. (2.7,0);
\shade[left color=black!40!white,right color=black!5!white](3.3,0)..controls (4,0).. (6,2)--(6,0);
\draw[thick](6,0)--(0,0);
\draw[thick](0,0)--(0,2)..controls (2,0).. (2.7,0);
\draw[thick](3.3,0)..controls (4,0).. (6,2)--(6,0);

\end{tikzpicture}

\caption{Possible shapes for a metric measure space of the type  $(X_{f},\di_\infty, \mathfrak{m}_{f,K})$ with $f\in \overbar{\mathscr{F}_k}$.}
\label{fig:example2}
\end{center}
\end{figure}
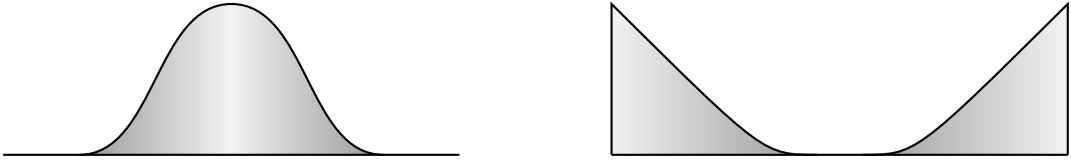

\subsection{Stronger $\cd$ conditions}\label{sec:conc2}

There are different ways to require convexity of the entropy along Wasserstein geodesics and they originate different $\cd$ conditions. In particular, the $\cd$ condition defined in Section \ref{sec:CD} is the weakest version, because it just requires the existence of a geodesic along which the entropy is convex. The most natural strengthening brings to the definition of the strong $\cd$ condition.

\begin{definition}
    A metric measure space $(X,\di,\m)$ is said to be a strong $\cd(0,\infty)$ space (or to satisfy the strong $\cd(0,\infty)$ condition) if the entropy functional $\Ent$ (see \eqref{eq:BSEnt}) is convex along every constant speed $W_2$-geodesic connecting two measures $\mu_0,\mu_1 \in \Prob_{ac}(X,\m)$.
\end{definition}

\begin{remark}
    The strong $\cd$ condition is usually defined with the dimensional parameter equal to $\infty$ for a technical reason related to Remark \ref{rmk:Ent}.
\end{remark}

This strengthening of the $\cd$ condition is sufficient to show the following result on the geodesic structure of the space, which was proved by Rajala and Sturm in \cite{rajalasturm}. 

\begin{theorem}\label{thm:RajalaSturm}
Every strong $\CD(0,\infty)$ metric measure space $(X,\di,\m)$ is essentially non-branching, i.e. for every pair $\mu_0,\mu_1\in\Prob_{ac}(X,\m)$, every $\eta\in \Prob(\Geo(X))$ representing a geodesic (see Remark \ref{rmk:W2geo}) connecting them, is concentrated on a non-branching set of geodesics.
\end{theorem} 

\noindent It is easy to realize that the same result cannot be true for $\cd$ spaces, an example for this is the metric measure space $(\R^2, \di_\infty, \Leb^2)$, which is a $\cd(0,2)$ space but it is not an essentially non-branching space. The same example shows that the strong $\cd$ condition is not stable with respect to the measured Gromov-Hausdorff convergence, and this constitutes a major flaw. It is interesting to wonder whether there exists a stronger version of the $\cd$ condition, which is stable with respect to the measured Gromov-Hausdorff convergence and still guarantees some additional properties on the space. To this purpose we consider the definition of the strict $\cd$ condition.

\begin{definition}\label{def:strictCD}
Given $N>1$, the metric measure space $(X,\di,\m)$ is said to be a strict $\cd(0,N)$ space (or to satisfy the strict $\cd(0,N)$ condition) if for every pair of measures $\mu_0,\mu_1 \in \Prob_{ac}(X,\m)$ there exists $\eta\in \Prob(\Geo(X))$ representing a constant speed $W_2$-geodesic $(\mu_t)_{t\in [0,1]}\subset \Prob_{ac}(X,\m)$ connecting them, such that, for every bounded measurable function $f : \Geo(X) \to \setR^+$ with $\int f \de \eta=1$ and every $N'\geq N$, the entropy functional $S_{N'}$ is convex along the geodesic $t \mapsto (e_t)_\# (f\eta)$.
\end{definition}

\noindent Analogously, it is possible to define the strict $\cd(0,\infty)$ condition, which was proved not to be stable with respect to the measured Gromov-Hausdorff convergence in \cite{MR4402722}. It is then not surprising that the same holds for the strict $\cd(0,N)$, in fact proceeding as in \cite{MR4402722} (see in particular Section 7) we can prove the following result.

\begin{prop}\label{prop:nonstab}
    Fix constant $k$ and $K$ such that the conclusions of Theorem \ref{thm:CDTheExample} and Corollary \ref{cor:main} hold with $N>1$.
    \begin{enumerate}
        \item For every $f\in \mathscr{F}_k$ the metric measure space $(X_f,\di_\infty, \m_{f,K})$ satisfy the strict $\cd(0,N)$ condition.
        \item Given $f\in \overbar{\mathscr{F}_k}$ increasing and such that $\{f=0\}=[-1,0]$, the metric measure space $(X_f,\di_\infty, \m_{f,K})$ does not satisfy the strict $\cd(0,N)$ condition.
    \end{enumerate}
    In particular, according to what we did in the proof of Corollary \ref{cor:main}, we conclude that the strict $\cd(0,N)$ condition is not stable with respect to the measured Gromov-Hausdorff convergence.
\end{prop}

\noindent Observe also that point 2 in this last proposition, combined with Corollary \ref{cor:main}, shows that the strict $\cd(0,N)$ condition is actually strictly stronger than the (classical) $\cd(0,N)$ condition. 

It is also possible to find a requirement which is intermediate between the strong $\cd$ condition and the strict $\cd$ one. This is called very strict $\cd$ condition and was introduced and studied by Schultz in \cite{schultz2017existence}. The original definition \cite[Definition 1]{schultz2017existence} is given with the dimensional parameter equal to $\infty$, but it can be easily adapted to the finite dimensional case following Definition \ref{def:strictCD}. Generalizing the work of Rajala and Sturm, Schultz proved that the very strict $\cd$ condition is sufficient to deduce a non-branching property on the space. 

\begin{theorem}\label{thm:schultz}
Every very strict $\CD(0,\infty)$ space is weakly essentially non-branching, i.e. for every pair $\mu_0,\mu_1\in\Prob_{ac}(X,\m)$ there exists $\eta\in \Prob(\Geo(X))$ representing a geodesic connecting them, that is concentrated on a non-branching set of geodesics. 
\end{theorem}

\noindent Being weakly essentially non-branching, is possibly the weakest meaningful non-branching condition that can be defined in metric measure spaces. However, as a consequence of Corollary \ref{cor:main}, we can observe that is not satisfied in every $\cd(0,N)$ space. In fact, taking $f\in \overbar{\mathscr{F}_k}$ increasing and such that $\{f=0\}=[-1,0]$, it is not difficult to see that every $\eta\in \Prob(\Geo(X))$, representing a geodesic connecting a marginal $\mu_0\in \Prob_{ac}(X,\m)$, concentrated in the one dimensional part $X_f \cap ([-1,0]\times \R)$, and a marginal $\mu_1\in \Prob_{ac}(X,\m)$, concentrated in the two dimensional part $X_f \cap ([0,1]\times \R)$, cannot be concentrated on a non-branching set of geodesics. Moreover, it is possible to observe that it cannot exist an optimal transport plan between $\mu_0$ and $\mu_1$, which is induced by a map. This proves that the essentially non-branching assumption is necessary to guarantee the existence of an optimal transport map, between absolutely continuous marginals in $\cd(0,N)$ spaces (cfr. \cite{MR2984123,rajalasturm,MR3691502,magnabosco2021optimal}). 

Finally, we point out that Proposition \ref{prop:nonstab} suggest that also the the very strict $\cd$ condition should not be stable with respect to the measured Gromov-Hausdorff convergence. However, in \cite{MR4432355} a stability result for the very strict $\cd$ condition is proved, under some additional metric assumptions on the converging sequence and on the limit space.

\subsection{Non-compact version}\label{sec:conc3}

In this last section we present a non-compact example of a singular $\cd(0,N)$, which structure is substantially analogous to the one of the space considered in Section \ref{sec:conc1}. In particular, given $k$ sufficiently small, introduce the metric measure space $(\X_k, \di_\infty, \m_{k,K})$ defined as
\begin{equation*}
    \X_k = \{0\}\times(-\infty,0] \,\, \cup \, \, \{(x,y)\in \setR^2\suchthat x \in [0,+\infty) \text{ and } 0\leq y\leq kx\} =: L \cup C^k,
\end{equation*}
\begin{equation*}
    \m_{k,K} := \mathbbm{1}_{\{x\leq 0\}}\cdot  C_K \cdot\mathcal   H^1 |_{\{y=0\}} + \mathbbm{1}_{\{x>0\}} \cdot\frac{1}{kx} \exp\left(-K \left(\frac{y}{kx} \right)^2\right)\cdot \Leb^2|_{X_f},
\end{equation*}
where $C_{K}= \int_0^1 e^{-K y^2} \de y$.

\begin{figure}[h]
\begin{center}

\begin{tikzpicture} 

\shade[left color=black!37!white,right color=black!5!white] (0,0)--(5.5,0)--(5.5,2.75)--(0,0)--cycle;
\draw[thick](-5,0)--(5,0);
\draw[dashed, thick](5,0)--(6,0);
\draw[dashed, thick](-5,0)--(-6,0);
\draw[thick](0,0)--(5,2.5);
\draw[dashed,thick](5,2.5)--(6,3);
\filldraw[black] (0,0) circle (1pt);

%\filldraw[white] (-7,0) circle (1pt);
\node at (3,1.7)[label=north:${y=kx}$] {};
\node at (-3,0)[label=north:${L}$] {};
\node at (4,0.5)[label=north:${C^k}$] {};
\end{tikzpicture}

\caption{The metric measure space $(\X_{k},\di_\infty, \mathfrak{m}_{k,K})$.}
\label{fig:example12}
\end{center}
\end{figure}
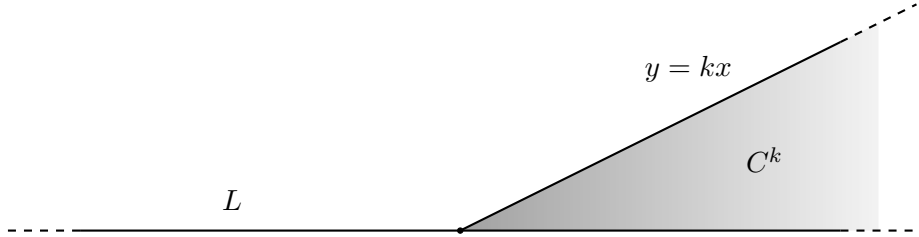

\begin{prop}\label{prop:non-compact}
    For suitable constants $k$, $K$ and $N>1$, the metric measure space $(\X_k, \di_\infty, \m_{k,K})$ satisfies the $\cd(0,N)$ condition.
\end{prop}

\begin{proof}[Sketch of the proof]
    With the exact same strategy developed in Section \ref{section:CDproof}, it is possible to prove that, for suitable constants $k$, $K$ and $N>1$, $(C^k,\di_\infty, \m_{k,K}|_{C^k})$ is a $\cd(0,N)$ space. In fact, all the steps can be repeated with minor changes and they give the same results.  In particular, every computation in Section \ref{section:CDproof} that needed the assumption $f< 3k$ in \ref{eq:defFk}, can be done also in this specific case, taking advantage of the fact that $(kx)''=0$. On the other hand, the space $(L,\di_\infty, \m_{k,K}|_{L})$ satisfies the $\cd(0,N)$ as well. 
    
    Now consider the full space $\X_k=L\cup C^k$, take a pair $\mu_0,\mu_1\in \Prob_{ac}(\X_k,\m_{k,K})$ and an optimal transport plan $\pi$ between them. The plan $\pi$ will send part of the mass of $\mu_0$ in $L$ in part of the mass of $\mu_1$ in $L$, part of the mass of $\mu_0$ in $C^k$ in part of the mass of $\mu_1$ in $C^k$ and part of the mass of $\mu_0$ in $L$ in part of the mass of $\mu_1$ in $C^k$ (or vice versa). It is possible to show that it is sufficient to prove entropy convexity on each of this ``sub-transport''. In particular, for the first two, this follows from the first part of the proof, therefore it is enough to prove entropy convexity for transports from $L$ to $C^k$. This is not trivial, but can be done finding a clever geodesic selection in accordance to the transport plan, following some ideas developed in \cite{ketterer2014failure} and \cite[Section 7]{MR4402722} (see also Remark 1 in \cite{ketterer2014failure}).
\end{proof}

\noindent Proposition \ref{prop:non-compact} allows to extend Theorem 3 of \cite{ketterer2014failure} to the setting of $\cd(0,N)$ spaces, proving the failure of topological splitting. Moreover, the metric measure space $(\X_k, \di_\infty, \m_{k,K})$ is invariant with respect to rescalings centred in the origin, thus it is the (unique) metric measure tangent of itself in the origin. Therefore $(\X_k, \di_\infty, \m_{k,K})$ provides an example of a $\cd(0,N)$ spaces having a metric measure tangent with a singular structure. 

%EDIT si può aggiungere il discorso su Le Donne e Magnabosco-Rossi

\bibliography{bibliografia}

\bibliographystyle{abbrv} 

\end{document}